\documentclass[12pt]{article}
\usepackage{amsfonts,amssymb,amsmath}
\usepackage{graphicx}
\numberwithin{equation}{section}

\newcommand{\beq}{ \begin{equation}}
\newcommand{\eeq}{ \end{equation}}
\newcommand{\bqa}{ \begin{eqnarray}}
\newcommand{\eqa}{ \end{eqnarray}}

\newtheorem{theorem}{Theorem}[section]
\newtheorem{dfnt}{Definition}[section]
\newtheorem{corollary}{Corollary}[section]
\newtheorem{prop}{Proposition}[section]

\newtheorem{prob}{Problem}[section]

\begin{document}
\title{Single Polygon Counting for $m$ Fixed Nodes in Cayley Tree: Two Extremal Cases}

\maketitle

\normalsize

\begin{center}
Farrukh Mukhamedov\footnote{far75m@yandex.ru}, \  Chin Hee Pah
\footnote{pahchinhee@hotmail.com} \ \
Mansoor  Saburov\footnote{msaburov@gmail.com}\\
Department of  Computational and Theoretical Sciences,\\
Faculty of Science, IIUM, Kuantan, Malaysia
\end{center}


\vskip 0.1cm
\abstract{  We denote a polygon as a connected
component in Cayley tree of order 2 containing certain number of fix
vertices. We found an exact formula for a polygon counting problem
for two cases, in which, for the first case the polygon contain a
full connected component of a Cayley tree and  for the second case
the polygon contain  two fixed vertices. From these formulas, which
is in the form of finite linear combination of Catalan numbers,  one
can find the asymptotic estimation for a counting problem.}

\vskip 2mm
 \noindent {\bf Keyword}: Cayley tree, connected component, Catalan number. \\
 Mathematics Subject Classification : 68M10, 11B83,  94C15 \\


\section{Introduction}

In the study of computer science, networks appears very often and
attracted a lot of attention of researchers \cite{mm,gdm2}. In a
generic undirected network particle can enter or exit at any
arbitrary sites. Again, the presence of loops in the generic
networks, also make the study of particle transport difficult. A
prototype network is a Cayley tree \cite{bax}, where the direction of
transport, the entry and exit points are well defined. Absence of
loops make the study relatively simpler. Again, several physical
systems like, water transport in trees, transport of antibody in
idiotypic networks in immune system \cite{ANP}, and air circulation
in lung \cite{BB} are strikingly similar to this model system of
Cayley trees.

There are extensively many combinatorial problems on Cayley tree
e.g. the connection between prefix ordered sequences and rooted
labelled trees \cite{sunik, k, hp,ss} and Dyck path \cite{barcucci}.
It is not surprise that most of them are related to the  well known
Catalan numbers \cite{stan}. Catalan numbers is one of the most
frequently encounter integer sequence in counting problem \cite{hp}.
Nowadays, there are many applications as well as  generalizations of
such numbers \cite{A,HLM,MS}.

In this paper, we will consider the following problem: \textit{to
find the number of all different connected components of a Cayley
tree with $n$ number of vertices, containing the given $m$ number of
vertices (where $n\ge m$).}  We borrow the term using in integer
lattice, i.e. polygon, for this connected component. In computer
science, one can simply visualize this scenario as we are setting up
$m$ routers and each router is expands in a rooted tree, how many
different way we can arrange the network with given $n$ (nodes). As
pointed out in \cite{gdm2}, a giant connected component is analogous
to the percolation cluster in condensed matter. The size
distribution of these finite connected components also describe the
topology of a random network. So, we see this research not merely
as a mathematical exercise but it does provide some applications in
network theory as well.

Note that in \cite{pah}, the posted problem was solved for the case
$m=1,$ namely, it was shown that the number of different connected
component containing a fixed  root $x_0 \in V$ in a semi-infinite
Cayley tree for a given $n$ number of vertices is exactly the
Catalan number
\begin{equation}\label{cat}
C_n=\frac{1}{n+1}\binom{2n}{n}, \quad n \in {\mathbb N}.
\end{equation}

In this paper, first we give a topological structure of the
connected component of the Cayley tree. Using these structures, we
consider two extremal cases for single polygon counting problem. A
recurrent formula for the said problem is derived combinatorically,
then an explicit form of single polygon counting is found in term of
linear combination of Catalan numbers  where  the coefficients do
not depend on $n$. The linear combination formulas are derived using
generating vectors, analogues to generating function, which will be
defined later. From these formulas, one can easily find the
asymptotic behavior of the derived numbers.



\begin{center}
\section{A topological structure of the connected  component of the Cayley tree}
\end{center}

Recall that a  Cayley tree of order $k$ \cite{bax}, denoted as
$\Gamma^k$, is a  graph with no cycles, each vertex emanates $k+1$
edges. We denote the set of all vertices as $V$ and the set of all
edges as $E$, i.e. $\Gamma^k=(V,E)$. In this paper we restrict
ourselves to the Cayley tree of order 2, i.e. $\Gamma^2$.  Two
vertices $x,y\in V$ are called {\it nearest neighbors} if $\langle
x,y\rangle\in E.$ Let $K=(V_K,E_K)$ be a finite connected component
of a Cayley tree, where $V_K$ contains at least two vertices of the
tree.

For a given $x\in V_K$ we put
\begin{eqnarray*}
{\mathcal{N}}_K(x)=\{y\in V_K : \langle x,y\rangle\in E_K\}.
\end{eqnarray*}
The elements of ${\mathcal{N}}_K(x)$ are called the \textit{nearest
neighbors} of $x$ in $K.$ We stand $|{\mathcal{N}}_K(x)|$ for the
number of the elements of ${\mathcal{N}}_K(x).$

\begin{dfnt}
A vertex $x\in V_K$ is called a boundary of $K,$ if
$|{\mathcal{N}}_K(x)|=1$ i.e. the nearest neighbor of $x$ in $K$ is
one. Otherwise, it is called an interior vertex of $K.$
\end{dfnt}

By $\partial K$ and $intK$ we denote the set all of boundary and interior vertices of $K,$ respectively. \\

\noindent {\it Remark 1.} Note that the given definition of the
boundary and interior is totaly different from the ordinary
definitions of boundary and interior points of a set, which is used
in topological spaces or in the graph theory.

So, by definition, we have
\begin{eqnarray*}
\partial K&=&\{x\in V_K : |{\mathcal{N}}_K(x)|=1\},\\
int K &=&\{x\in V_K : |{\mathcal{N}}_K(x)|\ge 2\},
\end{eqnarray*}
and
\begin{eqnarray*}
\partial K\cap\ int K=\emptyset,
\quad \quad  \partial K\cup\ int K=V_K.
\end{eqnarray*}

\begin{dfnt}
A finite connected component $K=(V_K,E_K)$ is said to be full in a Cayley tree, if $|{\mathcal{N}}_K(x)|=3$ for any $x\in int K.$
\end{dfnt}

\noindent {\it Remark 2.}
If $int K=\emptyset$ then by definition we say that $K$ is full.

\begin{prop}\label{minimalconnectedcom}
For any finite connected component $K=(V_K,E_K)$ there exists one and only one full connected component
$\overline{K}=(V_{\bar{K}},E_{\bar{K}}),$ containing $K$ and contained in an arbitrary full connected component
$K^{'}=(V_{K^{'}},E_{K^{'}})$ which contains $K.$
\end{prop}

\begin{proof}
Assume that $int K\neq \emptyset$, otherwise nothing to prove. Let
us construct the following connected component
$\overline{K}=(V_{\bar{K}},E_{\bar{K}})$ of the tree:
\begin{eqnarray*}
V_{\bar{K}}=\{\bar{x}\in V : \exists x\in int K, \langle x,\bar{x}\rangle\in E\}=\bigcup\limits_{x\in int K}{\mathcal{N}}_E(x),
\end{eqnarray*}
where ${\mathcal{N}}_E(x)$ is a set all the nearest neighbors of $x$
in the tree.

From this construction one can see that
\begin{eqnarray*}
int \overline{K}=int K, \quad \partial\overline{K}\supset\partial K, \quad |{\mathcal{N}}_{\bar{K}}(x)|=3, \quad  x\in int \overline{K},
\end{eqnarray*}
which means $\overline{K}=(V_{\bar{K}}, E_{\bar{K}})$ is a full
connected component containing $K.$ It follows from this
construction that $\overline{K}=(V_{\bar{K}}, E_{\bar{K}})$
contained  any full connected component
$K^{'}=(V_{K^{'}},E_{K^{'}})$ which contains $K.$

Let us prove the uniqueness. Suppose the contrary, i.e. there exists
$\widetilde{K}=(V_{\tilde{K}}, E_{\tilde{K}})$ which satisfies the
assertion of the proposition. Then we have
\begin{eqnarray*}
int K\subset int \widetilde{K}\subset int \overline{K}\subset int K.
\end{eqnarray*}
Therefore $int \widetilde{K}= int \overline{K}.$ Since
$\widetilde{K}$ and $\overline{K}$ are full, then we have $\partial
\widetilde{K}=\partial \overline{K}.$ This means $\widetilde{K}=
\overline{K}$, which completes the proof. \end{proof}

A full connected component $\overline{K}=(V_{\bar{K}},E_{\bar{K}})$
which satisfies the assertion of Proposition
(\ref{minimalconnectedcom}) is called the \textit{minimal full
component over  $K=(V_K,E_K).$}

\begin{theorem}
A finite connected component $K=(V_K,E_K)$ is full if and only if $|\partial K|-|int K|=2.$
\end{theorem}
\begin{proof} \textit{Only if Part.}
Let  $K=(V_K,E_K)$ be a full connected component. Then we will prove that $|\partial K|-|int K|=2$. We use mathematical induction w.r.t. the number $|int K|$ of interior vertices of $K.$

Let $int K=\emptyset.$ Then  $K=\langle x,y\rangle\in E$ and
$\partial K=\{x,y\}.$ Therefore we have $|\partial K|-|int
K|=2-0=2.$ We suppose that the assertion of Theorem is true for any
full connected component $K$ with $|int K|\le k,$ i.e. $|\partial
K|-|int K|=2.$

Now assume that $|int K|=k+1.$ Let us consider a full connected
subcomponent $\widetilde{K}=(V_{\tilde{K}},E_{\tilde{K}})$ of $K$
such that:
\begin{itemize}
  \item [(i)] $V_{\tilde{K}}\subset V_K,$ $E_{\tilde{K}}\subset E_K;$
  \item [(ii)] $int \widetilde{K}\subset int K$ and $|int \widetilde{K}|=k$ i.e. $int \widetilde{K}=\{x_1,x_2,\cdots,x_k\};$
  \item [(iii)] If $int K\setminus int \widetilde{K}=\{x_{k+1}\}$ then we have $int \widetilde{K}\cap {\mathcal{N}}_K(x_{k+1})=1.$
\end{itemize}
From the construction one can see that
\begin{eqnarray*}
x_{k+1}\in \partial \widetilde{K}, \quad \partial \widetilde{K}\setminus\{x_{k+1}\}\subset \partial K, \quad |\partial K\setminus\partial \widetilde{K}|=2.
\end{eqnarray*}
Then we have  $|\partial K|=|\partial \widetilde{K}|+1$ and $|int
K|=|int \widetilde{K}|+1.$ Since $\widetilde{K}$ is a full connected
component with $|int \widetilde{K}|=k$, so according to the
assumption of induction we have $|\partial \widetilde{K}|-|int
\widetilde{K}|=2.$ Consequently, one finds  \begin{eqnarray*}
|\partial K|-|int K|= |\partial \widetilde{K}|+1-|int
\widetilde{K}|-1=|\partial \widetilde{K}|-|int \widetilde{K}|=2.
\end{eqnarray*}

\textit{If part.} Let $K=(V_K,E_K)$ be a connected component with
$|\partial K|-|int K|=2.$ We shall show that $K$ is full. To do it,
let us consider the minimal full connected component
$\overline{K}=(V_{\bar{K}},E_{\bar{K}})$ containing $K=(V_K,E_K).$
We then have $int K=int \overline{K}$ and $\partial K\subset
\partial \overline{K}$ (see Proposition \ref{minimalconnectedcom}).
Since $\overline{K}$ is full, it then follows from only if part of
this Theorem that $|\partial\overline{K}|-|int\overline{K}|=2.$
Therefore,
\begin{eqnarray*}
|\partial\overline{K}|=|int \overline{K}|+2=|int K|+2=|\partial K.|
\end{eqnarray*}
Finiteness of $\partial\overline{K}$ and $\partial K\subset
\partial \overline{K}$ imply  $\partial K= \partial \overline{K}.$
This means $K=\overline{K}$, which completes the proof.
\end{proof}

\begin{corollary}
For any connected component $K=(V_K,E_K)$ of a Cayley tree we have
\begin{eqnarray}\label{estmationforboundary}
2\le|\partial K|\le|int K|+2.
\end{eqnarray}
\end{corollary}

\begin{center}
\section{Polygon counting problem for $m$ vertices}
\end{center}

Suppose that  an arbitrary  $m$ number of vertices of a Cayley tree
$\Gamma^2=(V,E)$ be given. Namely $V_m=\{x_1,x_2,\dots,x_m\}$. Let
us state our problem:

\begin{prob}
Find the number of all connected components of a Cayley tree with $n$ number of vertices containing the given $m$ number of vertices, where $n\ge m$.
\end{prob}

Let
\begin{eqnarray*}
G_{n,V_m}^{m}=\{K=(V_K,E_K) : |V_K|=n, \quad V_K\supset V_m\}
\end{eqnarray*}
be the set  of all  connected components containing the given $m$
vertices $V_m.$ Our main task is to evaluate the number
$\left|G_{n,V_m}^{m}\right|$ of elements of $G_{n,V_m}^{m}.$ Since
there  always exists a connected component containing the given $m$
vertices in a Cayley tree.  Among such kind of connected components
we take minimal one, i.e. by a minimal connected component
containing vertices $V_m$ we mean a connected component
$K(V_m)=(V_K,E_K)$ such that $V_m\subset V_K$ and if $V_m$ is
contained in another connected component $K'=(V',E')$ then one has
$V_K\subset V'$ and $E_K\subset E'$. One can see that if
$K(x_i,x_j)$ is a shortest path connecting two vertices $x_i$ and
$x_j$ then
\begin{eqnarray*}
K(V_m)=\bigcup\limits_{i=1}^{m-1}K(x_i,x_{i+1}).
\end{eqnarray*}
In this case, we can reformulate our problem as follows:
\begin{prob}
Find the number of all connected components of a Cayley tree with
$n$ number of vertices, containing the given connected component
$K(V_m).$
\end{prob}
In other words, if
\begin{eqnarray*}
G_{n,K(V_m)}^m=\{K=(V_K,E_K): K\supset K(V_m),\quad |V_K|=n\}
\end{eqnarray*}
is the set all of the connected components containing the given connected component $K(V_m)$ then we have
$
\left|G_{n,V_m}^{m}\right|=\left|G_{n,K(V_m)}^{m}\right|.
$
It is easy to check that if
\begin{eqnarray*}
G_{n,K(V_m)}^{|\partial{K(V_m)}|}=\{K=(V_K,E_K): V_K\supset \partial K(V_m),\quad |V_K|=n\}
\end{eqnarray*}
is the set all of the connected components containing the given  number $|\partial K(V_m)|$ of vertices $\partial K(V_m),$ then we have
$$
\left|G_{n,K(V_m)}^{|\partial K(V_m)|}\right|=\left|G_{n,K(V_m)}^{m}\right|=\left|G_{n,V_m}^{m}\right|.
$$

Therefore, in what follows, we shall calculate the number of
elements of $G_{n,K}^{|\partial{K}|}$ which is the set all of the
connected components containing the given number $|\partial K|$ of
vertices $\partial K.$ From the inequality
\eqref{estmationforboundary} we can get estimation to the number
$|\partial K|$ as follows
\[2\le|\partial K|\le|int K|+2.\]

In next sections, we will solve the polygon calculating problem for
two extremal cases of the number $|\partial K|,$ namely we shall
consider cases: $|\partial K|=|int K|+2$ and $|\partial K|=2.$  In
other words, in the first case $K$ is a full connected component and
in the second case $K$ is a shortest path which connects  two given
vertices.

\begin{center}
\section{Polygon counting problem for a full connected component}
\end{center}

Let $K=(V_K,E_K)$ be a given full connected component of a Cayley
tree with $|\partial K|=m,$ where $m\ge 2.$ We then know that $|int
K|=m-2.$ By $F_{n,m-2}^{(m)}$ we denote the number of elements of
$G_{n,K}^{|\partial K|}$.

\begin{figure}[h]
\centering
\includegraphics[width=0.5\columnwidth, clip]{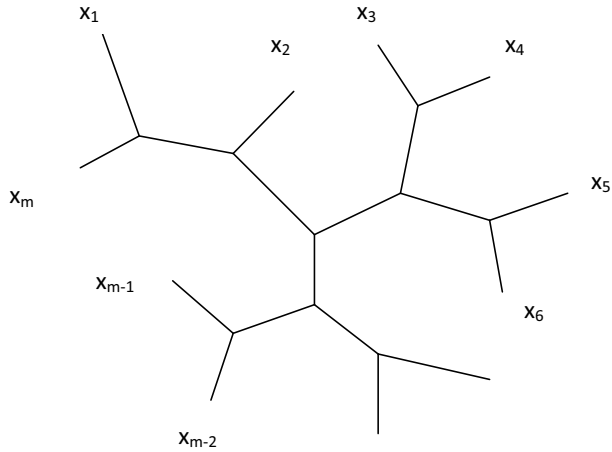}\\
\caption{A connected component containing vertices $x_1$, $x_2$,...,$x_m$, connected component which emanate out from $x_i$ is not shown.}
\end{figure}

\begin{prop}
Let $C_n$ be Catalan numbers, where $n\in {{\mathbb{N}}}$ (see
\eqref{cat}) then we heve
\begin{eqnarray}\label{fomulaforfulcomponent}
F_{n,m-2}^{(m)}=\sum\limits_{\substack{r_1+\cdots+r_m=n-m+2 \\r_1,\ldots, r_m\ge 1}}C_{r_1}\cdot\ldots \cdot C_{r_m}.
\end{eqnarray}
\end{prop}

\begin{proof}

Since there exists only unique full connected component with boundary $\partial
K$ which connects all $x_i$'s, (see Figure 1) the number of vertices
in the full connected component $m-2$ is always counted, left only
$n-m+2$ vertices for different configuration.

\begin{figure}[thbp]
\centering
\includegraphics[width=0.5\columnwidth, clip]{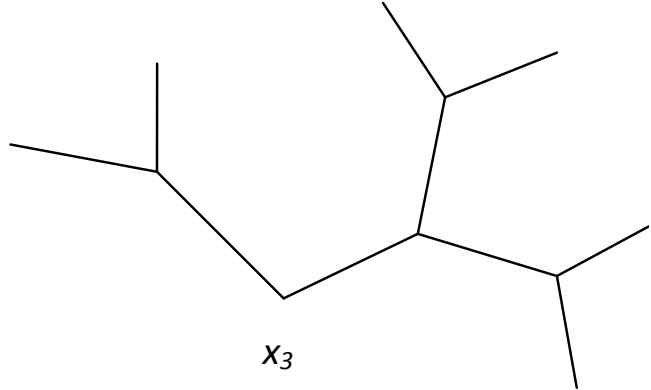}\\
\caption{A connected component containing vertices $x_3$, there is $C_{r_i}$ different connected component contains $x_3$ with $r_i$ number of vertices \cite{pah}.}
\end{figure}

For each vertex  $x_i \in \partial K$, let  $r_i$ denote the number of vertices emanate out from $x_i$, but not back to full connected component. It is known from \cite{pah}, there allow  $C_{r_i}$ number of different connected component in each box for given $r_i$ vertices (See Figure 2). The number of combination for all fixed $r_i $ over $\partial K$   is the product of $C_{r_1}C_{r_2}\cdots C_{r_{m-1}}C_{r_m}$ (For illustration, see Figure 3).
\begin{figure}[h]
\centering
\includegraphics[width=0.5\columnwidth, clip]{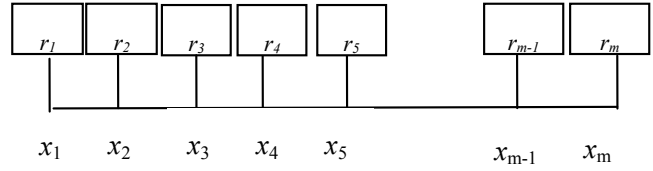}\\
\caption{Each $x_i$ is contained in a connected component with $r_i$ vertices.}
\label{PM}
\end{figure}
The total number is then the products sum over different combination
of  $r_1+r_2+\cdots+r_{m-1}+r_{m}+m-2=n$  i.e. the sum of all
vertices equals to $n$.  Note that  the vertices $x_i$ in the full
connected component are always occupied, therefore $r_i$  is
always at least 1.  Theorem is proved.
\end{proof}
\vskip 2mm

Let us calculate $F_{n,m-2}^{(m)}$ for small $m.$ For this we will
use the following recurrence formula for the Catalan number
\begin{eqnarray*}
C_n=\sum\limits_{i=0}^{n-1}C_iC_{n-1-i},
\end{eqnarray*}
here $C_0:=1.$

$1^\circ.$ Let $m=2.$ We then have
\begin{eqnarray}\label{formulaFnm=2}
F_{n,0}^{(2)}&=&\sum\limits_{\substack{r_1+r_2=n\\ r_1,r_2\ge1}}C_{r_1}C_{r_2}=\sum\limits_{r_1=1}^{n-1}C_{r_1}C_{n-r_1}=C_{n+1}-2C_0C_n\nonumber\\
&=&C_{n+1}-2C_{n}.
\end{eqnarray}
Let us list the first few terms of $F_{n,0}^{(2)}$ (see Sloane \cite{sloane} A002057):
\[1,4,14,48,165,572,2002,7072,25194,...\]

$2^\circ.$ Let $m=3.$ We then get
\begin{eqnarray}\label{formulaFnm=3}
F_{n,1}^{(3)}&=&\sum\limits_{\substack{r_1+r_2+r_3=n-1\nonumber\\ r_1,r_2,r_3\ge1}}C_{r_1}C_{r_2}C_{r_3}=\sum\limits_{r_1=1}^{n-3}C_{r_1}\sum\limits_{\substack{r_2+r_3=n-1-r_1\\r_2,r_3\ge1}}C_{r_2}C_{r_3}\nonumber\\
&=&\sum\limits_{r_1=1}^{n-3}C_{r_1}F_{n-1-r_1,0}^{(2)}=\sum\limits_{r_1=1}^{n-3}C_{r_1}\left(C_{n-r_1}-2C_{n-1-r_1}\right)=\nonumber\\
&=&\sum\limits_{r_1=1}^{n-1}C_{r_1}C_{n-r_1}-C_1C_{n-1}-C_2C_{n-2}-2\left(\sum\limits_{r_1=1}^{n-2}C_{r_1}C_{n-1-r_1}-C_1C_{n-2}\right)\nonumber\\
&=&\sum\limits_{\substack{r_1+r_2=n\\r_1,r_2\ge1}}C_{r_1}C_{r_2}-2\sum\limits_{\substack{r_1+r_2=n-1\\r_1,r_2\ge1}}C_{r_1}C_{r_2}-C_1C_{n-1}
=F_{n,0}^{(2)}-2F_{n-1,0}^{(2)}-C_{n-1}\nonumber\\
&=&C_{n+1}-4C_{n}+3C_{n-1}.
\end{eqnarray}

The first few terms of $F_{n,1}^{(3)}$ (see Sloane \cite{sloane} A003517) are as follows:
\[1,6,27,110,429,1638,6188,23256,87210,...\]

$3^\circ.$ Let $m=4.$ By means the formulas for $F_{n,1}^{(3)},$ $F_{n,0}^{(2)}$  we could get the following formula for $F_{n,2}^{(4)}$ after  algebraic manipulations
\begin{eqnarray}\label{formulaFnm=4}
F_{n,2}^{(4)}&=&\sum\limits_{\substack{r_1+\cdots+r_4=n-2\\ r_1,\ldots,r_4\ge1}}C_{r_1}\cdots C_{r_4}=\sum\limits_{r_1=1}^{n-5}C_{r_1}\sum\limits_{\substack{r_2+r_3+r_4=n-2-r_1\\r_2,r_3,r_4\ge1}}C_{r_2}C_{r_3}C_{r_4}\nonumber\\
&=&\sum\limits_{r_1=1}^{n-5}C_{r_1}F_{n-2-r_1,1}^{(3)}=\cdots=C_{n+1}-6C_{n}+10C_{n-1}-4C_{n-2}.
\end{eqnarray}
The first few terms of $F_{n,2}^{(4)}$ (see Sloane \cite{sloane} A003518) are as follows:
\[1,8,44,208,910,3808,15504,62016,245157,...\]

These calculations lead us to the following conjecture: the number
$F_{n,m-2}^{(m)}$ is represented as a linear combination of the
Catalan numbers.

Note that such a representation allows us to find asymptotical
estimation for $F_{n,m-2}^{(m)}$. Below we want to realize the
stated conjecture.

In the sequel, for any given vector $ a=(a_1,a_2,\cdots,a_m)\in
{\mathbb{R}}^{m}$ we put
\[  a_{\uparrow}=(a_1,a_2\cdots,a_m,0),\]
clearly  $ a_{\uparrow}\in{\mathbb{R}}^{m+1}.$

Let
\begin{equation}\label{matrixA} A_{m}=\left(
      \begin{array}{ccccccc}
        1 & 0 & 0 & 0 & \cdots & 0 & 0 \\
        -2C_0 & 1 & 0 & 0 & \cdots & 0 & 0 \\
        -C_1 & -2C_0 & 1 & 0 & \cdots & 0 & 0 \\
        -C_2 & -C_1 & -2C_0 & 1 & \cdots & 0 & 0 \\
        \cdots & \cdots & \cdots & \cdots & \cdots & \cdots & \cdots \\
        -C_{m-3} & -C_{m-4} & -C_{m-5} & -C_{m-6} & \cdots & 1 & 0 \\
        -C_{m-2} & -C_{m-3} & -C_{m-4} & -C_{m-5} & \cdots & -2C_0 & 1 \\
      \end{array}
    \right)
\end{equation}
and
\begin{equation}\label{matrixB}
B_{[m-1\times m]}=\left(
                  \begin{array}{ccccccc}
                    -C_{m} & -C_{m-1} & -C_{m-2} & \cdots & -C_{3} & -C_{2} & -C_{1} \\
                    -C_{m+1} & -C_{m} & -C_{m-1} & \cdots & -C_{4} & -C_{3} & -C_{2} \\
                    -C_{m+2} & -C_{m+1} & -C_{m} & \cdots & -C_{5} & -C_{4} & -C_{3} \\
                    \cdots & \cdots & \cdots & \cdots & \cdots & \cdots & \cdots \\
                    -C_{2m-2} & -C_{2m-3} & -C_{2m-4} & \cdots & -C_{m+1} & -C_{m} & -C_{m-1} \\
                  \end{array}
                \right)
\end{equation}
be $m\times m$ and $(m-1)\times m$ matrices, respectively (where $m\ge 2$).

Let us consider the following vectors
\begin{eqnarray*}
a^{(1)}=1,\quad a^{(2)}=A_{2}a_{\uparrow}^{(1)},\quad a^{(3)}=A_{3}a_{\uparrow}^{(2)}, \quad  \ldots, \quad a^{(m+1)}=A_{m+1}a_{\uparrow}^{(m)}.
\end{eqnarray*}
\begin{theorem} Let $A_{m}$ and $B_{[m-1\times m]}$ be  the matrices given by \eqref{matrixA} and \eqref{matrixB}, here $m\ge 2$. If $a^{(1)}=1,$ $a^{(m+1)}=A_{m+1}a_{\uparrow}^{(m)}$ and $B_{[m-1\times m]}a^{(m)}=\theta_{m-1}$ then $B_{[m\times m+1]}{a^{(m+1)}}=\theta_m,$ here $\theta_m=(0,0,\cdots,0)$ is  zero vector of ${\mathbb{R}}^{m}.$
\end{theorem}
\begin{proof}
We use mathematical induction w.r.t. $m\ge 2.$ Let $m=2.$ Since $a^{(1)}=1$ and
\[
a^{(2)}=A_{2}a_{\uparrow}^{(1)}=\left(
                                          \begin{array}{cc}
                                            1 & 0 \\
                                            -2C_0 & 1 \\
                                          \end{array}
                                        \right)\binom{1}{0}
=(1,-2C_0),\]
\[
B_{[1\times2]}a^{(2)}=(-C_2,-C_1)\binom{1}{-2C_0}=0=\theta_1,
\]
we obtain
\[
a^{(3)}=A_{3}a_{\uparrow}^{(2)}=\left(
                                          \begin{array}{ccc}
                                            1 & 0 & 0\\
                                            -2C_0 & 1 & 0 \\
                                            -C_1 & -2C_0 & 1 \\
                                          \end{array}
                                        \right)\left(\begin{array}{c}
                                                 1 \\
                                                 -2C_0 \\
                                                 0
                                               \end{array}
                                               \right)
=(1,-4C_0,-C_1+4C_0^2)\]
\[
B_{[2\times3]}a^{(3)}=\left(
                      \begin{array}{ccc}
                        -C_3 & -C_2 & -C_1 \\
                        -C_4 & -C_3 & -C_2 \\
                      \end{array}
                    \right)\left(
                             \begin{array}{c}
                               1 \\
                               -4C_0 \\
                               -C_1+4C_0^2 \\
                             \end{array}
                           \right)
                    =(0,0)=\theta_2.
\]
We assume that for $m=k-1$ the statement of the Theorem is true, i.e., $B_{[k-1\times k]}a^{(k)}=\theta_{k-1}$ which is equivalent
\begin{eqnarray}\label{assumptionfork-1}
-\sum\limits_{j=1}^kC_{k+i-j}a_j^{(k)}=0
\end{eqnarray}
where $i=\overline{1,k-1},$ $j=\overline{1,k}$ and $a^{(k)}=(a_1^{(k)},\cdots,a_k^{(k)}).$

We will prove the statement of Theorem for $m=k.$ Let $D_{[k\times k+1]}=B_{[k\times k+1]}A_{k+1}$. We then have
\[B_{[k\times k+1]}a^{(k+1)}=B_{[k\times k+1]}A_{k+1}a_{\uparrow}^{(k)}=D_{[k\times k+1]}a_{\uparrow}^{(k)}.\]
Let us evaluate the matrix $D_{[k\times k+1]}.$ We know that
\[
A_{k+1}=(a_{ij})_{i,j=1}^{k+1}, \quad a_{ij}=\left\{\begin{array}{c}
                                                         0,\quad\quad i<j \\
                                                         1, \quad\quad i=j \\
                                                         -2C_0 \quad\quad i=j+1 \\
                                                         -C_{i-j-1} \quad i>j+1
                                                       \end{array}
                                                       \right.\\
\]
and
\[
B_{[k\times k+1]}=(b_{ij})_{i,j=1}^{k,k+1}=(-C_{k+1+i-j})_{i,j=1}^{k,k+1}
\]
Hence, we get
\begin{eqnarray*}
d_{ij}&=&\sum\limits_{l=1}^{k+1}b_{il}a_{lj}=\sum\limits_{l=j}^{k+1}b_{il}a_{lj}\\
&=&-C_{k+1+i-j}+2C_0C_{k+i-j}+C_1C_{k+i-j-1}+C_2C_{k+i-j-2}+\cdots+C_{k-j}C_{i}\\
&=&-\sum\limits_{l=0}^{k+i-j}C_{l}C_{k+i-j-l}+\sum\limits_{l=0}^{k-j}C_{l}C_{k+i-j-l}+C_{k+i-j}C_0\\
&=&-\sum\limits_{l=k-j+1}^{k+i-j}C_{l}C_{k+i-j-l}+C_{k+i-j}C_0
\end{eqnarray*}
If we set $r=l-k+j$ then one can have
\begin{eqnarray}\label{fordij}
d_{ij}&=&-\sum\limits_{r=1}^{i}C_{k-j+r}C_{i-r}+C_{k+i-j}C_0.
\end{eqnarray} where $i=\overline{1,k}$ and $j=\overline{1,k+1}.$
Therefore, it follows from \eqref{fordij} that
\begin{eqnarray}\label{formulaforDai}
\left(D_{[k\times k+1]}a_{\uparrow}^{(k)}\right)_i&=&\sum\limits_{j=1}^{k+1}d_{ij}\widetilde{a_{j}^{(k)}}=
\sum\limits_{j=1}^{k}d_{ij}{a_{j}^{(k)}}\nonumber\\
&=&\sum\limits_{j=1}^{k}\left(-\sum\limits_{r=1}^{i}C_{k-j+r}C_{i-r}+C_{k+i-j}C_0\right)a_j^{(k)}\nonumber\\
&=&\sum\limits_{r=1}^{i}C_{i-r}\left(-\sum\limits_{j=1}^{k}C_{k+r-j}a_j^{(k)}\right)+C_0\sum\limits_{j=1}^{k}C_{k+i-j}a_j^{(k)}
\end{eqnarray} where $i=\overline{1,k}.$
Consequently, from \eqref{formulaforDai} and \eqref{assumptionfork-1} one finds:

$(i)$ If $i=1$ then
\[\left(D_{[k\times k+1]}a_{\uparrow}^{(k)}\right)_1=-C_0\sum\limits_{j=1}^{k}C_{k+1-j}a_j^{(k)}+C_0\sum\limits_{j=1}^{k}C_{k+1-j}a_j^{(k)}=0;\]

$(ii)$ If $1<i\le k$ then
\[\left(D_{[k\times k+1]}a_{\uparrow}^{(k)}\right)_i=\sum\limits_{r=1}^{i-1}C_{i-r}\left(-\sum\limits_{j=1}^{k}C_{k+r-j}a_j^{(k)}\right)=0.
\]
This means that $D_{[k\times k+1]}a_{\uparrow}^{(k)}=\theta_{k}$ and completes the proof.
\end{proof}

\begin{theorem}\label{analyticfomulaforFnm-2}
Let $A_{m}$ be  the matrix given by \eqref{matrixA} and $a^{(1)}=1,$ $a^{(m+1)}=A_{m+1}a_{\uparrow}^{(m)}$. Let $C_n$ be the Catalan number. Then $$F_{n,m-2}^{m}=a_1^{(m)}C_{n+1}+a_{2}^{(m)}C_n+\cdots+a_m^{(m)}C_{n-m+2},$$ where $a^{(m)}=(a_1^{(m)},a_2^{(m)},\cdots,a_m^{(m)}).$
\end{theorem}
\begin{proof}
We use mathematical induction w.r.t. $m\ge2.$ Let $m=2.$ Since $a^{(1)}=1$ and
\[
a^{(2)}=A_{2}a_{\uparrow}^{(1)}=\left(
                                          \begin{array}{cc}
                                            1 & 0 \\
                                            -2C_0 & 1 \\
                                          \end{array}
                                        \right)\binom{1}{0}
=(1,-2C_0),\] from \eqref{formulaFnm=2} we obtain that
\[F_{n,0}^{(m)}=C_{n+1}-2C_{n}=1\cdot C_{n+1}-2C_0\cdot C_{n}=a_{1}^{(2)}C_{n+1}+a_{2}^{(2)}C_{n}.\]
We suppose that the statement of the Theorem is true for $m=k$, i.e.,
\begin{eqnarray}\label{assumptionfork}
F_{n,k-2}^{(k)}=a_1^{(k)}C_{n+1}+a_{2}^{(k)}C_n+\cdots+a_k^{(k)}C_{n-k+2},
\end{eqnarray}
and we  prove for $m=k+1.$
Then, it follows from \eqref{fomulaforfulcomponent} that
\begin{eqnarray*}
F_{n,k-1}^{(k+1)}&=&\sum\limits_{\substack{r_1+\cdots+r_{k+1}=n-k+1 \\r_1,\ldots, r_{k+1}\ge 1}}C_{r_1}\cdot\ldots \cdot C_{r_{k+1}}\\
&=&\sum\limits_{r_1=1}^{n-2k+1}C_{r_1}\sum\limits_{\substack{r_2+\cdots+r_{k+1}=n-k+1-r_1 \\r_2,\ldots, r_{k+1}\ge 1}}C_{r_2}\cdot\ldots \cdot C_{r_{k+1}}=\sum\limits_{r_1=1}^{n-2k+1}C_{r_1}F_{n-1-r_1,k-2}^{(k)}
\end{eqnarray*}
Using \eqref{assumptionfork}, from the last equality we get
\begin{eqnarray*}
F_{n,k-1}^{(k+1)}&=&\sum\limits_{r_1=1}^{n-2k+1}C_{r_1}\sum\limits_{j=1}^{k}a_{j}^{(k)}C_{n+1-r_1-j}
=\sum\limits_{j=1}^{k}a_{j}^{(k)}\sum\limits_{r_1=1}^{n-2k+1}C_{r_1}C_{n+1-r_1-j}\\
&=&\sum\limits_{j=1}^{k}a_{j}^{(k)}\left(\sum\limits_{r_1=1}^{n-j}C_{r_1}C_{n+1-r_1-j}-\sum\limits_{r_1=n-2k+2}^{n-j}C_{r_1}C_{n+1-r_1-j}\right)\\
&=&\sum\limits_{j=1}^{k}a_{j}^{(k)}\left(F_{n+1-j,0}^{(2)}-\sum\limits_{r_1=n-2k+2}^{n-j}C_{r_1}C_{n+1-r_1-j}\right)
\end{eqnarray*}
If we set $i=n+1-j-r_1$ then using \eqref{formulaFnm=2} one gets
\begin{eqnarray*}
F_{n,k-1}^{(k+1)}&=&\sum\limits_{j=1}^{k}a_{j}^{(k)}\left(F_{n+1-j,0}^{(2)}-\sum\limits_{i=1}^{2k-1-j}C_{i}C_{n+1-j-i}\right)\\
&=&\sum\limits_{j=1}^{k}a_{j}^{(k)}\left(C_{n+2-j}-2C_{0}C_{n+1-j}-\sum\limits_{i=1}^{2k-1-j}C_{i}C_{n+1-j-i}\right)\\
&=&a_{1}^{(k)}\left(C_{n+1}-2C_{0}C_{n}-\sum\limits_{i=1}^{2k-2}C_{i}C_{n-i}\right)\\
&&+a_2^{(k)}\left(C_{n}-2C_{0}C_{n-1}-\sum\limits_{i=1}^{2k-3}C_{i}C_{n-1-i}\right)+\cdots\\
&&+a_{l}^{(k)}\left(C_{n+2-l}-2C_{0}C_{n+1-l}-\sum\limits_{i=1}^{2k-1-l}C_{i}C_{n+1-l-i}\right)+\cdots\\
&&+a_{k}^{(k)}\left(C_{n+2-k}-2C_{0}C_{n+1-k}-\sum\limits_{i=1}^{k-1}C_{i}C_{n+1-k-i}\right)
\end{eqnarray*}
follow by
\begin{eqnarray*}
F_{n,k-1}^{(k+1)}&=&a_{1}^{(k)}C_{n+1}+\left(a_{2}^{(k)}-2C_{0}a_{1}^{(k)}\right)C_{n}+\left(a_{3}^{(k)}-2C_0a_{2}^{(k)}-C_{1}a_{1}^{(k)}\right)C_{n-1}\\
&&+\left(a_{4}^{(k)}-2C_0a_{3}^{(k)}-C_1a_{2}^{(k)}-C_2a_{1}^{(k)}\right)C_{n-2}+\cdots\\
&&+\left(a_{l}^{(k)}-2C_0a_{l-1}^{(k)}-C_{1}a_{l-2}^{(k)}-\cdots-C_{l-2}a_{1}^{(k)}\right)C_{n-l+2}+\cdots\\
&&+\left(a_{k}^{(k)}-2C_0a_{k-1}^{(k)}-C_{1}a_{k-2}^{(k)}-\cdots-C_{k-2}a_{1}^{(k)}\right)C_{n-k+2}\\
&&+\left(-2C_0a_{k}^{(k)}-C_1a_{k-1}^{(k)}-C_{2}a_{k-2}^{(k)}-\cdots-C_{k-1}a_{1}^{(k)}\right)C_{n-k+1}\\
&&+\left(-C_1a_{k}^{(k)}-C_2a_{k-1}^{(k)}-C_{3}a_{k-2}^{(k)}-\cdots-C_{k}a_{1}^{(k)}\right)C_{n-k}+\cdots\\
&&+\left(-C_{k-1}a_{k}^{(k)}-C_{k}a_{k-1}^{(k)}-C_{k+1}a_{k-2}^{(k)}-\cdots-C_{2k-2}a_{1}^{(k)}\right)C_{n-2k+2}\\
&=&\left(A_{k+1}a_{\uparrow}^{(k)}\right)_1C_{n+1}+\cdots+\left(A_{k+1}a_{\uparrow}^{(k)}\right)_{k+1}C_{n-k+1}\\
&&+\left(B_{[k-1\times k]}a^{(k)}\right)_1C_{n-k}+\cdots+\left(B_{[k-1\times k]}a^{(k)}\right)_{k-1}C_{n-2k+2}
\end{eqnarray*}
Taking into account
\begin{eqnarray*}
A_{k+1}a_{\uparrow}^{(k)}=a^{(k+1)},\quad \quad  B_{[k-1\times k]}a^{(k)}=\theta_{k-1}.
\end{eqnarray*}
we obtain
\[F_{n,k-1}^{(k+1)}=a_1^{(k+1)}C_{n+1}+a_{2}^{(k+1)}C_n+\cdots+a_{k+1}^{(k+1)}C_{n-k+1},\]
and this completes the proof.
\end{proof}

\noindent
{\it Remark 3.}
Since the vector $a^{(m)}$ does not depend on $n$ and all coordinates of the vector $a^{(m)}$ are nonzero, one can find an estimate easily through the following well-known estimation of Catalan numbers
\begin{eqnarray}\label{estimationofCatalan}
C_n=O(1)\frac{4^n}{n^{3/2}}
\end{eqnarray}
\begin{corollary}
The asymptotic estimation of $F_{n,m-2}^{m}$ is given by
\[F_{n,m-2}^{m}=O(1)C_{n-m+2}=O(1)\frac{4^{n-m+2}}{(n-m+2)^{3/2}}. \]
\end{corollary}

\begin{center}
\section{Polygon counting problem for two vertices}
\end{center}

Let $x_1$ and $x_m$ be two  given vertices and $K(x_1,x_m)=(V_m,E_m)$ be the shortest path connecting the given vertices $x_1$ and $x_m.$ We then know that $|\partial K|=\{x_1,x_m\}.$ We suppose that $|V_m|=m$ and $m\ge 2.$ Then it is clear that $|int K(x_1,x_m)|=m-2.$ By $T_{n,m}^{(2)}$ we denote the number of elements of $G_{n,K(x_1,x_m)}^{|\partial K(x_1,x_m)|}$ which is the set all of the connected component, with $n$ number of vertices, containing two given vertices $x_1$ and $x_m$.\\

\begin{prop}
Let $C_n$ be Catalan numbers, where $n\in {{\mathbb{N}}},$ then we heve
\begin{eqnarray}\label{fomulafortwovertices}
T_{n,m}^{(2)}=\sum\limits_{\substack{r_1+\cdots+r_m=n-m+2 \\{r_1,r_m\ge 1}, r_2\ldots,r_{m-1}\ge0}}C_{r_1}\cdot\ldots \cdot C_{r_m}.
\end{eqnarray}
\end{prop}
\begin{proof}
Since vertices  $x_1$ and $x_m$  always present (see Figure 3) and connected by a unique shortest path with m vertices, we denote this set of vertices by $V_m$, the collection of $x_1$,   $x_m$ and vertices along shortest path. For each vertex  $x_i \in V_m$, there allow $C_{r_i}$  number of different connected component, which the connected component not in the $V_m$,  emanate from  $x_i$  with $r_i$  number of vertices (see Figure 4).

\begin{figure}[thbp]
\centering
\includegraphics[width=0.5\columnwidth, clip]{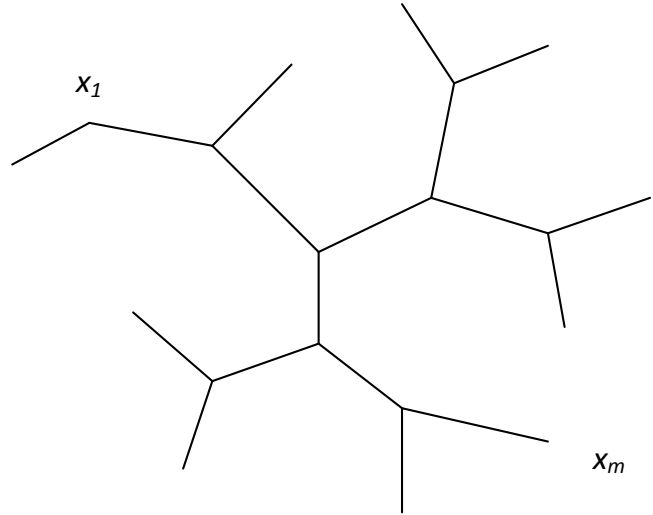}\\
\caption{A connected component containing vertices $x_1$ and $x_m$ only, connected component which emanate out from them is not shown.}
\label{PM}
\end{figure}

The number of combination for each fix $x_i$  over  $V_m$ is the product of $C_{r_1}\cdot C_{r_2}\cdots C_{r_m}.$ The total number is then sum over  $r_1+r_2+\cdots+r_{m-1}+r_m+m-2=n$ which excluding the $m-2$ vertices. Since $x_1$ and $x_m$  always present, $r_1$ and $r_m$  is always at least one. Theorem proved.

\vskip 2mm
\begin{figure}[thbp]
\centering
\includegraphics[width=0.5\columnwidth, clip]{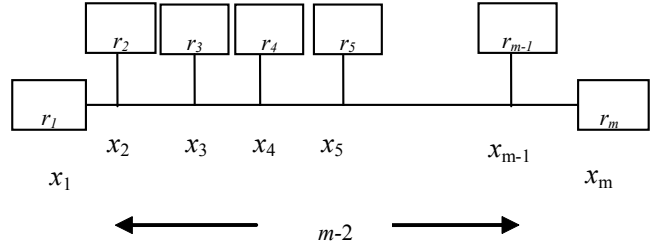}\\
\caption{There are $m-2$ vertices along short path, excluding $x_1$ and $x_2$. From each $x_i$  emanate another single connected component with $r_i$ vertices.}
\end{figure}
\vskip 2mm

\end{proof}
\vskip 2mm

Let us calculate $T_{n,m}^{(2)}$ for small $m.$

$1^\circ.$ Let $m=2.$ We then have
\begin{eqnarray}\label{formulaTnm=2}
T_{n,2}^{(2)}&=&\sum\limits_{\substack{r_1+r_2=n\\ r_1,r_2\ge1}}C_{r_1}C_{r_2}=C_{n+1}-2C_{n}.
\end{eqnarray}

$2^\circ.$ Let $m=3.$ We then get
\begin{eqnarray}\label{formulaTnm=3}
T_{n,3}^{(2)}&=&\sum\limits_{\substack{r_1+r_2+r_3=n-1\nonumber\\ {r_1,r_3\ge1,r_2\ge0}}}C_{r_1}C_{r_2}C_{r_3}=\sum\limits_{r_2=0}^{n-3}C_{r_2}\sum\limits_{\substack{r_1+r_3=n-1-r_2\\r_1,r_3\ge1}}C_{r_1}C_{r_3}\nonumber\\
&=&\sum\limits_{r_2=0}^{n-3}C_{r_2}T_{n-1-r_2,2}^{(2)}=\sum\limits_{r_2=0}^{n-3}C_{r_2}\left(C_{n-r_2}-2C_{n-1-r_2}\right)\nonumber\\
&=&\sum\limits_{r_2=0}^{n}C_{r_2}C_{n-r_2}-C_0C_{n}-C_1C_{n-1}-C_2C_{n-2}\nonumber\\
&&-2\left(\sum\limits_{r_2=0}^{n-1}C_{r_2}C_{n-1-r_2}-C_0C_{n-1}-C_{1}C_{n-2}\right)\nonumber\\
&=&C_{n+1}-3C_{n}+C_{n-1}.
\end{eqnarray}

$3^\circ.$ Let $m=4.$ After  a little bit algebraic manipulation we could get the following formula for $T_{n,4}^{(2)}$
\begin{eqnarray}\label{formulaTnm=4}
T_{n,4}^{(2)}&=&\sum\limits_{\substack{r_1+\cdots+r_4=n-2\\ {r_1,r_4\ge1,r_2,r_3\ge0}}}C_{r_1}\cdots C_{r_4}=\sum\limits_{r_2=0}^{n-4}C_{r_2}\sum\limits_{\substack{r_1+r_3+r_4=n-2-r_2\\{r_1,r_4\ge1,r_3\ge0}}}C_{r_1}C_{r_3}C_{r_4}\nonumber\\
&=&\sum\limits_{r_2=0}^{n-4}C_{r_2}T_{n-1-r_1,3}^{(2)}=\cdots=C_{n+1}-4C_{n}+3C_{n-1}.
\end{eqnarray}

If we denote by $T_{n,1}^{(2)}:=C_{n+1}-C_{n}$ for the sake of the beauty of the formula, then one can observe the following recurrence formula
\begin{eqnarray}\label{recformulaforT3}
T_{n,3}^{(2)}&=&T_{n,2}^{(2)}-T_{n-1,1}^{(2)}\\
T_{n,4}^{(2)}&=&T_{n,3}^{(2)}-T_{n-1,2}^{(2)}
\end{eqnarray}

\noindent
{\it Remark 4.}
It is clear that $T_{m,m}^{(2)}=1$ for any $m\ge2.$

Let us show this recurrence formula for $T_{n,m}^{(2)}$ in general case.

\begin{theorem}
We have the following recurrence formula for $T_{n,m}^{(2)}$
\begin{eqnarray}\label{recformulaforTnm}
T_{n,m}^{(2)}=T_{n,m-1}^{(2)}-T_{n-1,m-2}^{(2)}
\end{eqnarray}
where $n>m$ and $m\ge3.$
\end{theorem}
\begin{proof}
We use a mathematical induction w.r.t. $m.$ Let $m=3.$ In \eqref{recformulaforT3} it is shown  that the recurrence formula \eqref{recformulaforTnm} is true. We suppose that the recurrence formula \eqref{recformulaforTnm} is true for any  $m\le k$ and we  prove for $m=k+1.$
We know that
\begin{eqnarray*}
T_{n,k+1}^{(2)}&=&\sum\limits_{\substack{r_1+\cdots+r_{k+1}=n-k+1 \\{r_1,r_{k+1}\ge 1}, r_2\ldots,r_{k}\ge0}}C_{r_1}\cdot\ldots \cdot C_{r_{k+1}}\\
&=&\sum\limits_{r_2=0}^{n-k-1}C_{r_2}\sum\limits_{\substack{r_1+r_3+\cdots+r_{k+1}=n-k+1-r_2 \\{r_1,r_{k+1}\ge 1}, r_3\ldots,r_{k}\ge0}}C_{r_1}\cdot C_{r_3}\ldots \cdot C_{r_{k+1}}=\sum\limits_{r_2=0}^{n-k-1}C_{r_2}T^{(2)}_{n-1-r_2,k},
\end{eqnarray*} for any $n> k+1$ and $k\ge2.$ Then due to assumption of the induction one gets
\begin{eqnarray*}
T_{n,k+1}^{(2)}&=&\sum\limits_{r_2=0}^{n-k-1}C_{r_2}\left(T^{(2)}_{n-1-r_2,k-1}-T^{(2)}_{n-2-r_2,k-2}\right)\\
&=&\sum\limits_{r_2=0}^{n-k-1}C_{r_2}T^{(2)}_{n-1-r_2,k-1}-\sum\limits_{r_2=0}^{n-k-1}C_{r_2}T^{(2)}_{n-2-r_2,k-2}\\
&=&\sum\limits_{r_2=0}^{n-k}C_{r_2}T^{(2)}_{n-1-r_2,k-1}-C_{n-k}T^{(2)}_{k-1,k-1}-\sum\limits_{r_2=0}^{n-k}C_{r_2}T^{(2)}_{n-2-r_2,k-2}
+C_{n-k}T^{(2)}_{k-2,k-2}\\
&=&T_{n,k}^{(2)}-T_{n-1,k-1}^{(2)}.
\end{eqnarray*}
This completes the proof.
\end{proof}

Let us try to find an analytic formula for calculating $T_{n,m}^{(2)}.$

\begin{theorem}
There exist vectors \[a^{\lfloor m\rfloor}=\left(a_1^{\lfloor m\rfloor},-a_2^{\lfloor m\rfloor},\cdots,(-1)^{l-1}a_l^{\lfloor m\rfloor},\cdots,(-1)^{{\lfloor m\rfloor}-1}a_{\lfloor m\rfloor}^{\lfloor m\rfloor}\right)\]
such that $a_l^{\lfloor m\rfloor}>0$ for any $l=\overline{1,\lfloor m\rfloor}$ and
\begin{eqnarray*}\label{analyticformulaforTnm}
T_{n,m}^{(2)}=a_1^{\lfloor m\rfloor}C_{n+1}-a_2^{\lfloor m\rfloor}C_{n}+\cdots+(-1)^{l-1}a_l^{\lfloor m\rfloor}C_{n+2-l}+\cdots+(-1)^{{\lfloor m\rfloor}-1}a_{\lfloor m\rfloor}^{\lfloor m\rfloor}C_{n+2-\lfloor m\rfloor},
\end{eqnarray*}
here $\lfloor m\rfloor:=[\frac{m-1}{2}]+2.$
\end{theorem}

\begin{proof}
We use a mathematical induction w.r.t. $m.$ In \eqref{formulaTnm=2}, \eqref{formulaTnm=3}, \eqref{formulaTnm=4} for $m=2,3,4,$ we have already found the vectors $a^{\lfloor 2\rfloor}, a^{\lfloor 3\rfloor}, a^{\lfloor 4\rfloor}$ satisfying the assertion of Theorem.

We suppose that the assertion  of Theorem is true for $m\le k$ and we  prove for $m=k+1.$
Due to assumption of induction we have
\begin{eqnarray*}
T_{n,k}^{(2)}&=&a_1^{\lfloor k\rfloor}C_{n+1}-a_2^{\lfloor k\rfloor}C_{n}+\cdots+(-1)^{{\lfloor k\rfloor}-1}a_{\lfloor k\rfloor}^{\lfloor k\rfloor}C_{n+2-\lfloor k\rfloor},\\
T_{n-1,k-1}^{(2)}&=&a_1^{\lfloor k-1\rfloor}C_{n}-a_2^{\lfloor k-1\rfloor}C_{n-1}+\cdots+(-1)^{{\lfloor k-1\rfloor}-1}a_{\lfloor k-1\rfloor}^{\lfloor k-1\rfloor}C_{n+1-\lfloor k-1\rfloor}.
\end{eqnarray*}
We know that if $k=2t$ then $\lfloor k+1\rfloor=t+2,$ $\lfloor k\rfloor=\lfloor k-1\rfloor=t+1$ and \[n+1-\lfloor k-1\rfloor=n+2-\lfloor k+1\rfloor=n-t,\]
and if $k=2t+1$ then $\lfloor k+1\rfloor=\lfloor k\rfloor=t+2,$ $\lfloor k-1\rfloor=t+1$ and
\[n+1-\lfloor k-1\rfloor=n+2-\lfloor k\rfloor=n+2-\lfloor k+1\rfloor=n-t.\]
Then, it follows from \eqref{recformulaforTnm} that
\begin{eqnarray*}
T_{n,k+1}^{(2)}&=&T_{n,k}^{(2)}-T_{n-1,k-1}^{(2)}\\
&=&a_1^{\lfloor k+1\rfloor}C_{n+1}-a_2^{\lfloor k+1\rfloor}C_{n}+\cdots+(-1)^{l-1}a_l^{\lfloor k+1\rfloor}C_{n+2-l}+\\
&&\quad\quad\quad\quad\quad\quad\quad\quad\quad\quad\quad\quad+\cdots+(-1)^{{\lfloor k+1\rfloor}-1}a_{\lfloor k+1\rfloor}^{\lfloor k+1\rfloor}C_{n+2-\lfloor k+1\rfloor},
\end{eqnarray*}
where
\begin{eqnarray}\label{recurenceformulaforam}
a_1^{\lfloor k+1\rfloor}&=&a_1^{\lfloor k\rfloor},\nonumber\\
a_2^{\lfloor k+1\rfloor}&=&a_2^{\lfloor k\rfloor}+a_1^{\lfloor k-1\rfloor},\nonumber\\
&\cdots&\nonumber\\
a_l^{\lfloor k+1\rfloor}&=&a_l^{\lfloor k\rfloor}+a_{l-1}^{\lfloor k-1\rfloor},\\
&\cdots&\nonumber\\
a_{\lfloor k+1\rfloor}^{\lfloor k+1\rfloor}&=&\left\{\begin{array}{cc}
                                                     a_{\lfloor k\rfloor}^{\lfloor k\rfloor}+a_{\lfloor k-1\rfloor}^{\lfloor k-1\rfloor}, &  k=2t+1\\
                                                     a_{\lfloor k-1\rfloor}^{\lfloor k-1\rfloor}, &   k=2t.
                                                     \end{array}
\right.\nonumber
\end{eqnarray}
Since $a_i^{\lfloor k\rfloor}>0$ and $a_j^{\lfloor k-1\rfloor}>0$ for any $i=\overline{1,\lfloor k\rfloor}$ and $j=\overline{1,\lfloor k-1\rfloor},$ we have $a_l^{\lfloor k+1\rfloor}>0$ for all $l=\overline{1,\lfloor k+1\rfloor}.$ This completes the proof of Theorem.
\end{proof}

Right now, we will try to find an explicit form of the vector
\begin{eqnarray*}
a^{\lfloor m\rfloor}&=&\left(a_1^{\lfloor m\rfloor},-a_2^{\lfloor m\rfloor},\cdots,(-1)^{l-1}a_l^{\lfloor m\rfloor},\cdots,(-1)^{{\lfloor m\rfloor}-1}a_{\lfloor m\rfloor}^{\lfloor m\rfloor}\right)
\end{eqnarray*} for any $m\ge 3.$ We  will call this vector as a \textit{generating vector} of $T_{n,m}^{(2)}.$

Let us introduce the following notations for a given vector $a^{(m)}=(a_1,\cdots,a_m)\in{\mathbb{R}}^{m-1}$
\begin{eqnarray*}
a^{(m)}_{\downarrow}=(0,a^{(m)})=(0,a_1^{(m)},\cdots,a_m^{(m)}),\\
a^{(m)}_{\Lsh}=\left\{\begin{array}{cc}
               (a^{(m)},0) & \lfloor m+1\rfloor>\lfloor m\rfloor \\
               a^{(m)} & \lfloor m+1\rfloor=\lfloor m\rfloor
             \end{array}\right.
\end{eqnarray*} here $(a^{(m)},0)=(a_1^{(m)},\cdots,a_m^{(m)},0)$ and $\lfloor m\rfloor=[\frac{m-1}{2}]+2.$

Let
\begin{equation*}\label{matrixT} {\mathcal{T}}_{m}=\left(
      \begin{array}{ccccccc}
        1 & 0 & 0 & 0 & \cdots & 0 & 0 \\
        -C_0 & 1 & 0 & 0 & \cdots & 0 & 0 \\
        -C_1 & -C_0 & 1 & 0 & \cdots & 0 & 0 \\
        -C_2 & -C_1 & -C_0 & 1 & \cdots & 0 & 0 \\
        \cdots & \cdots & \cdots & \cdots & \cdots & \cdots & \cdots \\
        -C_{m-3} & -C_{m-4} & -C_{m-5} & -C_{m-6} & \cdots & 1 & 0 \\
        -C_{m-2} & -C_{m-3} & -C_{m-4} & -C_{m-5} & \cdots & -C_0 & 1 \\
      \end{array}
    \right)
\end{equation*}
be $m\times m$ matrix.

\begin{theorem}
Let $a^{\lfloor m+1\rfloor}$ and $a^{\lfloor m\rfloor}$ be the generating vectors of $T_{n,m+1}^{(2)}$ and $T_{n,m}^{(2)},$ respectively. Then we have the following formula
\begin{eqnarray}\label{am+1Tm+1am}
a^{\lfloor m+1\rfloor}&=& {\mathcal{T}}_{\lfloor m+1\rfloor}a^{\lfloor m\rfloor}_\Lsh
\end{eqnarray} where ${\lfloor m\rfloor}=[\frac{m-1}{2}]+2$ and $m\ge 2.$
\end{theorem}

\begin{proof} We know that if $a^{\lfloor m+1\rfloor},$ $a^{\lfloor m\rfloor},$ $a^{\lfloor m-1\rfloor}$ are the generating vectors of $T_{n,m+1}^{(2)},$ $T_{n,m}^{(2)},$ $T_{n,m-1}^{(2)},$ respectively then it follows from \eqref{recurenceformulaforam} that
\begin{eqnarray}\label{compactformofam}
a^{\lfloor m+1\rfloor}=a^{\lfloor m\rfloor}_\Lsh-a^{\lfloor m-1\rfloor}_\downarrow
\end{eqnarray}

We use mathematical induction w.r.t. $m\ge2$ to prove the formula \eqref{am+1Tm+1am}.

For $m=2,3,$ the formula \eqref{am+1Tm+1am} is true. We suppose that the formula \eqref{am+1Tm+1am} is true for $m\le k$ and we prove for $m=k+1.$ Here, it can be happened two case either $\lfloor k+1\rfloor=\lfloor k\rfloor$ or $\lfloor k+1\rfloor>\lfloor k\rfloor.$ We will consider the first case, i.e. $\lfloor k+1\rfloor=\lfloor k\rfloor$. Analogously, one can show for the second case.

Since $\lfloor k+1\rfloor=\lfloor k\rfloor$ due to \eqref{compactformofam} we have
\begin{eqnarray}\label{forak+1}
a^{\lfloor k+1\rfloor}=a^{\lfloor k\rfloor}_\Lsh-a^{\lfloor k-1\rfloor}_\downarrow=a^{\lfloor k\rfloor}-a^{\lfloor k-1\rfloor}_\downarrow.
\end{eqnarray}
It follows from \eqref{forak+1} and the assumption of induction that
\begin{eqnarray*}
a^{\lfloor k+1\rfloor}&=&{\mathcal{T}}_{\lfloor k\rfloor}a_\Lsh^{\lfloor k-1\rfloor}-a^{\lfloor k-1\rfloor}_\downarrow\\
&=&{\mathcal{T}}_{\lfloor k\rfloor}a_\Lsh^{\lfloor k\rfloor}+{\mathcal{T}}_{\lfloor k\rfloor}a_\Lsh^{\lfloor k-1\rfloor}-{\mathcal{T}}_{\lfloor k\rfloor}a_\Lsh^{\lfloor k\rfloor}-a^{\lfloor k-1\rfloor}_\downarrow\\
&=&{\mathcal{T}}_{\lfloor k\rfloor}a_\Lsh^{\lfloor k\rfloor}+{\mathcal{T}}_{\lfloor k\rfloor}\left(a_\Lsh^{\lfloor k-1\rfloor}-a_\Lsh^{\lfloor k\rfloor}\right)-a^{\lfloor k-1\rfloor}_\downarrow.
\end{eqnarray*}
Since $\lfloor k+1\rfloor=\lfloor k\rfloor$ we have ${\mathcal{T}}_{\lfloor k+1\rfloor}={\mathcal{T}}_{\lfloor k\rfloor}$ and
\begin{eqnarray*}
a^{\lfloor k+1\rfloor}&=&{\mathcal{T}}_{\lfloor k+1\rfloor}a_\Lsh^{\lfloor k\rfloor}+{\mathcal{T}}_{\lfloor k\rfloor}\left(a_\Lsh^{\lfloor k-1\rfloor}-a_\Lsh^{\lfloor k\rfloor}\right)-a^{\lfloor k-1\rfloor}_\downarrow.
\end{eqnarray*}
Right now we want to show that
\begin{eqnarray*}
{\mathcal{T}}_{\lfloor k\rfloor}\left(a_\Lsh^{\lfloor k-1\rfloor}-a_\Lsh^{\lfloor k\rfloor}\right)=a^{\lfloor k-1\rfloor}_\downarrow
\end{eqnarray*}
We know that $\lfloor k+1\rfloor=\lfloor k\rfloor.$ Then $a_\Lsh^{\lfloor k\rfloor}=a^{\lfloor k\rfloor}$ and
\begin{eqnarray*}
a_\Lsh^{\lfloor k-1\rfloor}-a_\Lsh^{\lfloor k\rfloor}=a_\Lsh^{\lfloor k-1\rfloor}-a^{\lfloor k\rfloor}=a_\downarrow^{\lfloor k-2\rfloor}.
\end{eqnarray*}
Since $\lfloor k-1\rfloor=\lfloor k-2\rfloor$ due to assumption of induction we have
\begin{eqnarray*}
a^{\lfloor k-1\rfloor}={\mathcal{T}}_{\lfloor k-1\rfloor}a_\Lsh^{\lfloor k-2\rfloor}={\mathcal{T}}_{\lfloor k-1\rfloor}a^{\lfloor k-2\rfloor}.
\end{eqnarray*}
From the construction of ${\mathcal{T}}_{\lfloor k\rfloor}$ it follows immediately that
\begin{eqnarray*}
{\mathcal{T}}_{\lfloor k\rfloor}\left(a_\Lsh^{\lfloor k-1\rfloor}-a_\Lsh^{\lfloor k\rfloor}\right)={\mathcal{T}}_{\lfloor k\rfloor}a_\downarrow^{\lfloor k-2\rfloor}=\left(0,{\mathcal{T}}_{\lfloor k-1\rfloor}a^{\lfloor k-2\rfloor}\right)=a_\downarrow^{\lfloor k-1\rfloor}.
\end{eqnarray*}
This completes the proof.
\end{proof}

Analogously, using the estimation \eqref{estimationofCatalan} of Catalan numbers, one finds the following asymptotical estimation for $T_{n,m}^{(2)}.$
\begin{corollary}
The asymptotic estimation of $T_{n,m}^{(2)}$ is given by
\[T_{n,m}^{(2)}=O(1)C_{n-\lfloor m\rfloor+2}=O(1)\frac{4^{n-\lfloor m\rfloor}+2}{(n-\lfloor m\rfloor+2)^{3/2}}, \]
where $\lfloor m\rfloor=[\frac{m-1}{2}]+2.$
\end{corollary}

\section{Conclusion}
We obtained an exact formula for $F_{n,m-2}^{m}$ and $T_{n,m}^{(2)}$ as a linear combination of Catalan numbers, where the coefficient does not depend on $n$. An estimate is then derived.

\vskip 0.5in
\noindent {\bf\Large Acknowledgments}\\

\noindent  This research is funded by the IIUM Research Endowment Fund EDW B 0905-296.

\end{document}